\documentclass[12pt,letterpaper]{article}

\usepackage{amssymb,amsfonts,amscd,amsthm}
\usepackage[all,arc]{xy}
\usepackage{enumerate, bbm}
\usepackage{mathrsfs}
\usepackage{tikz}
\usepackage[left=0.9in,top=0.9in,right=0.9in,bottom=0.9in]{geometry}
\usepackage{mathtools}
\usepackage{hyperref}
\usepackage{color}%colors
\usepackage{graphicx}
\usepackage{enumitem} %Allows resuming enumeration and numbering by section
\graphicspath{ {TexImages/} }

\newtheorem{thm}{Theorem}[section]

\newtheorem{prop}[thm]{Proposition}
\newtheorem{lem}[thm]{Lemma}
\newtheorem{claim}[thm]{Claim}

\theoremstyle{definition}
\newtheorem{defn}{Definition}

\hypersetup{
	colorlinks,
	citecolor=blue,
	filecolor=blue,
	linkcolor=blue,
	urlcolor=blue,
	linktocpage
}

\setenumerate[1]{label=\thesection.\arabic*.} %These two make your enumerated lists start with the section number
\setenumerate[2]{label*=\arabic*.} %Makes subparts labeled by numbers instead of letters.
\setlist[enumerate]{itemsep=2ex, topsep=2ex} %spaces out enumerate/itemize better
\setlist[itemize]{itemsep=2ex, topsep=2ex}

\newcommand{\E}{\mathbb{E}}

\newcommand{\al}{\alpha}

\newcommand{\del}{\delta}

\renewcommand{\l}{\left}
\renewcommand{\r}{\right}

\newcommand{\half}{\frac{1}{2}}

\newcommand{\sm}{\setminus}
\newcommand{\sub}{\subseteq}

\renewcommand{\c}[1]{\mathcal{#1}}
\renewcommand{\b}[1]{\mathbf{#1}}

\newcommand{\tr}[1]{\textrm{#1}}
\newcommand{\rec}[1]{\frac{1}{#1}}
\newcommand{\f}[2]{\frac{#1}{#2}}
\newcommand{\floor}[1]{\l\lfloor #1\r\rfloor}

\newcommand{\mr}[1]{\mathrm{#1}}

\newcommand{\Var}{\mr{Var}}

\newcommand{\cH}{\mathcal{H}}

\title{A Smoother Notion of Spread Hypergraphs}
\author{Sam Spiro\footnote{Dept.\ of Mathematics, UCSD {\tt sspiro@ucsd.edu}. This material is based upon work supported by the National Science Foundation Graduate Research Fellowship under Grant No. DGE-1650112.}}
\date{\today}
\begin{document}
	\maketitle
	
\begin{abstract}
	Alweiss, Lovett, Wu, and Zhang introduced $q$-spread hypergraphs in their breakthrough work regarding the sunflower conjecture, and since then $q$-spread hypergraphs have been used to give short proofs of several outstanding problems in probabilistic combinatorics.  A variant of $q$-spread hypergraphs was implicitly used by Kahn, Narayanan, and Park to determine the threshold for when a square of a Hamiltonian cycle appears in the random graph $G_{n,p}$.  In this paper we give a common generalization of the original notion of $q$-spread hypergraphs and the variant used by Kahn et al.
\end{abstract}

\section{Introduction}
This paper concerns hypergraphs, and throughout we allow our hypergraphs to have repeated edges.  If $A$ is a set of vertices of a hypergraph $\cH$, we define the \textit{degree of $A$} to be the number of edges of $\cH$ containing $A$, and we denote this quantity by $d_{\cH}(A)$, or simply by $d(A)$ if $\cH$ is understood.  We say that a hypergraph $\cH$ is \textit{$q$-spread} if it is non-empty and if $d(A)\le q^{|A|}|\cH|$ for all sets of vertices $A$.  A hypergraph is said to be \textit{$r$-bounded} if each of its edges have size at most $r$ and it is \textit{$r$-uniform} if all of its edges have size exactly $r$.

The notion of $q$-spread hypergraphs was introduced by Alweiss, Lovett, Wu, and Zhang \cite{Sunflower} where it was a key ingredient in their groundbreaking work which significantly improved upon the bounds on the largest size of a set system which contain no sunflower.  Their method was refined by Frankston, Kahn, Narayanan, and Park \cite{Fractional} who proved the following.
\begin{thm}[\cite{Fractional}]\label{thm:Fractional}
	There exists an absolute constant $K_0$ such that the following holds. Let $\cH$ be an $r$-bounded $q$-spread hypergraph on $V$.  If $W$ is a set of size $K_0 (\log r) q  |V|$ chosen uniformly at random from $V$, then $W$ contains an edge of $\cH$ with probability tending to 1 as $r$ tends towards infinity.
\end{thm}
This theorem was used in \cite{Fractional} to prove a number of remarkable results.  In particular it resolved a conjecture of Talagrand, and it also gave a much simpler solution to Shamir's problem, which had originally been solved by Johansson, Kahn, and Vu~\cite{JKV}.  

Kahn, Narayanan, and Park~\cite{Hamiltonian} used a variant of the method from \cite{Fractional} to show that for certain $q$-spread hypergraphs, the conclusion of Theorem~\ref{thm:Fractional} holds for random sets $W$ of size only $C q |V|$.  They used this to determine the threshold for when a square of a Hamiltonian cycle appears in the random graph $G_{n,p}$, which was a long-standing open problem.

In a talk,  Narayanan asked if there was a ``smoother'' definition of spread hypergraphs which interpolated between $q$-spread hypergraphs and hypergraphs like those in \cite{Hamiltonian} where the $\log r$ term of Theorem~\ref{thm:Fractional} can be dropped. The aim of this paper is to provide such a definition. 

\begin{defn}\label{defMain}
	Let $0<q\le 1$ be a real number and $r_1>\cdots >r_\ell$ positive integers. We say that a hypergraph $\cH$ on $V$ is \textit{$(q;r_1,\ldots,r_\ell)$-spread} if $\cH$ is non-empty, $r_1$-bounded, and if for all $A\sub V$ with $d(A)>0$ and $r_{i}\ge|A|\ge r_{i+1}$ for some $1\le i< \ell$, we have for all $j\ge r_{i+1}$ that
	\[M_j(A):=|\{S\in \cH:|A\cap S|\ge j\}|\le q^j |\cH|.\]
	%and for all $A\sub V$ with $d(A)>0$ and $|A|=r_\ell$, we have $d(A)=M_{r_\ell}(A)\le q^{r_\ell}|\cH|$.
\end{defn}

Roughly speaking, this condition says that every set $A$ of $r_i$ vertices intersects few edges of $\cH$ in more than $r_{i+1}$ vertices.   \iffalse Before stating our main result, let us work through a few basic computations concerning $(q;r_1,\ldots,r_\ell)$-spread hypergraphs to get a better feel for their properties.  We first establish a sufficient condition that seems to be useful in practice.
\begin{lem}
	Let $0<q\le 1$ be a real number and $r_1>\cdots >r_\ell$ positive integers.  If $\cH$ is a hypergraph on $V$ which is non-empty, $r_1$-bounded, and is such that for all $A\sub V$ of size $r_i$ and all $j\ge r_{i+1}$ we have
	\[\sum_{B\sub A: |B|=j} d(B)\le q^j|\cH|,\]
	then $\cH$ is $(q;r_1,\ldots,r_\ell)$-spread.
\end{lem}
\begin{proof}
	If $A'$ is a set with $r_i\ge |A|\ge r_{i+1}$ and $j\ge r_{i+1}$, then by selecting any set $A\supseteq A'$ of size $r_i$ we see that
	\[M_j(A')\le M_j(A)\le \sum_{B\sub A,\ |B|=j} d(B)\le q^j|\cH|,\]
	which proves the result. 
\end{proof}\fi 

As a warm-up, we show how this definition relates to the definition of being $q$-spread. %Here and throughout the document we omit floors and ceilings whenever they are not crucial to our analysis.

\begin{prop}\label{prop:spreadEquiv}
	We have the following.
	\begin{itemize}
		\item[(a)] If $\cH$ is $(q;r_1,\ldots,r_\ell,1)$-spread, then it is $q$-spread.  
		\item[(b)] If $\cH$ is $q$-spread and $r_1$-bounded, then it is $(4q;r_1,\ldots,r_\ell)$-spread for any sequence of integers $r_i$ satisfying $r_i>r_{i+1}\ge \half r_i$.
	\end{itemize}
\end{prop}
\begin{proof}
	For (a), assume $\cH$ is $(q;r_1,\ldots,r_\ell,1)$-spread and let $r_{\ell+1}=1$.  Let $A$ be a set of vertices of $\cH$.  If $A=\emptyset$, then $d(A)=|\cH|=q^{|A|}|\cH|$, so we can assume $A$ is non-empty.  If $d(A)=0$, then trivially $d(A)\le q^{|A|}|\cH|$, so we can assume $d(A)>0$.  This means $|A|\le r_1$ since in particular $\cH$ is $r_1$ bounded.   Thus there exists an integer $1\le i\le \ell$ such that  $r_i\ge |A|\ge r_{i+1}$, so the hypothesis that $\cH$ is $(q;r_1,\ldots,r_\ell,1)$-spread and $d(A)>0$ implies
	\[d(A)\le M_{|A|}(A)\le q^{|A|}|\cH|,\]
	proving that $\cH$ is $q$-spread.
	
	For (b), assume $\cH$ is $q$-spread and $r_1$-bounded. If $A$ is any set of vertices of $\cH$, then for all $j\ge \half |A|$ we have
	\[M_j(A)\le \sum_{B\sub A:|B|=j} d(B)\le 2^{|A|}\cdot  q^j |\cH|\le (4q)^j|\cH|.\]
	In particular, if $r_i\ge |A|\ge r_{i+1}$, then this bound holds for any $j\ge r_{i+1}$ since $r_{i+1}\ge \half r_i\ge \half |A|$. We conclude that $\cH$ is $(4q;r_1,\ldots,r_\ell)$-spread. 
\end{proof}

We now state our main result for uniform hypergraphs, which says that a random set of size $C\ell q |V|$ will contain an edge of an $r_1$-uniform $(q;r_1,\ldots,r_\ell,1)$-spread hypergraph with high probability as $C\ell$ tends towards infinity.    An analogous result can be proven for non-uniform hypergraphs, but for ease of presentation we defer this result to Section~\ref{sec:con}.
\begin{thm}\label{thm:spreadMain}
	There exists an absolute constant $K_0$ such that the following holds. Let $\cH$ be an $r_1$-uniform $(q;r_1,\ldots,r_\ell,1)$-spread hypergraph on $V$.  If $W$ is a set of size $C\ell q|V|$ chosen uniformly at random from $V$ with $C\ge K_0$, then
	\[\Pr[W\tr{ contains an edge of }\cH]\ge 1-\f{K_0}{C\ell}.\]
\end{thm}
We note that Theorem~\ref{thm:spreadMain} with $\ell=\Theta(\log r)$ together with Proposition~\ref{prop:spreadEquiv}(b) implies Theorem~\ref{thm:Fractional} for uniform $\cH$. In \cite{Hamiltonian}, it is implicitly proven that the hypergraph $\c{H}$ encoding squares of Hamiltonian cycles is a $(2n)$-uniform $(Cn^{-1/2};2n,C_0n^{1/2},1)$-spread hypergraph for some appropriate constants $C,C_0$, so the $\ell=2$ case of Theorem~\ref{thm:spreadMain} suffices to prove the main result of \cite{Hamiltonian}. Thus, at least in the uniform case, Theorem~\ref{thm:spreadMain} provides an interpolation between the results of \cite{Fractional,Hamiltonian}. Theorem~\ref{thm:spreadMain} can also be used to recover results from very recent work of Espuny D\'iaz and Person~\cite{DP} who extended the results of \cite{Hamiltonian} to other spanning subgraphs\footnote{Somewhat more precisely, let $\c{H}$ be the hypergraph whose hyperedges consist of copies of $F$ in $K_n$.  If $F$ has $r$ edges and maximum degree $d$, and if $\c{H}$ is $(q,\al,\del)$-superspread as defined in \cite{DP}, then one can show that $\c{H}$ is $(Cq;r,C_1r^{1-\al},C_2r^{1-2\al},\ldots,C_{\floor{1/\al}}r^{1-\floor{1/\al}\al},1)$-spread for some constants $C,C_i$ which depend on $d,\del$.  Indeed, when verifying Definition~\ref{defMain} for $j\ge \del k$,  one can use a similar argument as in Proposition~\ref{prop:spreadEquiv}(b) and the fact that $\c{H}$ is $q$-spread.  If $j<\del k$, then the superspread condition together with Lemma 2.3 of \cite{DP} can be used to give the result.}  of $G_{n,p}$.

\section{Proof of Theorem~\ref{thm:spreadMain}}
Our approach borrows heavily from Kahn, Narayanan, and Park~\cite{Hamiltonian}.  We break our proof into three parts: the main reduction lemma, auxiliary lemmas to deal with some special cases, and a final subsection proving the theorem.

\subsection{The Main Lemma}
We briefly sketch our approach for proving Theorem~\ref{thm:spreadMain}.  Let $\cH$ be a hypergraph with vertex set $V$.  We first choose a random set $W_1\sub V$ of size roughly $q|V|$.  If $W_1$ contains an edge of $\cH$ then we would be done, but most likely we will need to try and add in an additional random set $W_2$ of size $q|V|$ and repeat the process.  In total then we are interested in finding the smallest $I$ such that $W_1\cup \cdots \cup W_I$ contains an edge of $\cH$ with relatively high probability.  One way to guarantee that $I$ is small would be if we had $|S\sm W_1|$ small for most $S\in \cH$ (i.e., most vertices of most edges $S\in \cH$ are covered by $W_1$), and then that $W_2$ covered most of the vertices of most $S\sm W_1$, and so on.  

The condition that, say,  $|S\sm W_1|$ is small for most $S\in \cH$ turns out to be too strong a condition to impose.  However, if $\cH$ is sufficiently spread, then we can guarantee a weaker result: for most $S\in \cH$, there is an $S'\sub S\cup W_1$ such that $|S'\sm W_1|$ is small.  We can then discard $S$ and focus only on $S'$, and by iterating this repeatedly we obtain the desired result.

To be more precise, given a hypergraph $\cH$, we say that a pair of sets $(S,W)$ is \textit{$k$-good} if there exists $S'\in \cH$ such that $S'\sub S\cup W$ and $|S'\sm W|\le k$, and we say that the pair is \textit{$k$-bad} otherwise.   The next lemma shows that $(q;r,k)$-spread hypergraphs have few $k$-bad pairs with $S\in \cH$ and $W$ a set of size roughly $q|V|$.  In the lemma statement we adopt the notation that ${V\choose m}$ is the set of subsets of $V$ of size $m$.

\begin{lem}\label{lem:spreadMain}
	Let $\cH$ be an $r$-uniform $n$-vertex hypergraph on $V$ which is $(q;r,k)$-spread.  Let $C\ge 4$ and define $p=Cq$.  If $pn\ge 2r$ and $p\le \half$, then 
	
	\begin{equation*}\l|\l\{(S,W): S\in \cH,\ W\in {V\choose pn},\ (S,W)\tr{ is }k\tr{-bad}\r\}\r|\le 3(C/2)^{-k/2}|\cH| {n\choose pn}.\end{equation*}
\end{lem}
\begin{proof}
	Throughout this lemma we make frequent use of the identity
	\[{a-c\choose b-c}/{a\choose b}={b\choose c}/{a\choose c},\]
	which follows from the simple combinatorial identity ${a\choose b}{b\choose c}={a\choose c}{a-c\choose b-c}$.

	For $t\le r$, define \begin{align*} \c{B}_{t}=\{(S,W): S\in \cH,\ W\in {V\choose pn},\ (S,W)\tr{ is }&k\tr{-bad},\ |S\cap W|=t\}.\end{align*}
	Observe that the quantity we wish to bound is $\sum_{t} |\c{B}_{t}|$, so it suffices to bound each term of this sum.  From now on we fix some $t$ and define
	\[w=pn-t.\]
	
	At this point we need to count the number of elements in $\c{B}_t$, and there are several natural approaches that could be used.  One way would be to first pick any $S\in \cH$ and then count how many $W$ satisfy $(S,W)\in \c{B}_{t}$.  Another approach would be to pick any set $Z$ of size $r+w$ (which will be the size of $S\cup W$ since $|S\cap W|=t$) and then bound how many $S,W\sub Z$ have $(S,W)\in \c{B}_{t}$.  For some pairs the first approach is more efficient, and for others the second is.  In particular, the second approach will be more effective whenever $Z=S\cup W$ contains few elements of $\c{B}_{t}$.

	With this in mind, we say that a set $Z$ is \textit{pathological} if
	\[|\{S\in \cH: S\sub Z,\ (S,Z\sm S)\tr{ is }k\tr{-bad}\}|>N,\]
	where \[N:=(C/2)^{-k/2}|\cH|{n-r\choose w}/{n\choose w+r}=(C/2)^{-k/2}|\cH|{w+r\choose r}/{n\choose r}.\]
	We say that a pair $(S,W)$ is \textit{pathological} if the set $S\cup W$ is pathological and that $(S,W)$ is \textit{non-pathological} otherwise.
	
	\begin{claim}
		The number of $(S,W)\in \c{B}_{t}$ which are non-pathological is at most
		\begin{equation*}{n\choose r+w} N {r\choose t}=(C/2)^{-k/2}|\cH|{r\choose t}{n-r\choose w}.\label{eq:nonpathological}\end{equation*}
	\end{claim}
	\begin{proof}
		We identify each of the non-pathological pairs $(S,W)$ by specifying $S\cup W$, then $S$, then $S\cap W$.  
		
		Observe that $S\cup W$ is a non-pathological set of size $r+w$, and in particular there are at most ${n\choose r+w}$ ways to make this first choice.  Fix such a non-pathological set $Z$ of size $r+w$.  Observe that if $(S,W)$ is $k$-bad with $S\cup W= Z$, then $(S,Z\sm S)$ is also $k$-bad.  Because $Z$ is non-pathological, there are at most $N$ choices for $S$ such that $(S,Z\sm S)$ is $k$-bad.  Given $S$, there are at most ${r\choose t}$ choices for $S\cap W$.  Multiplying the number of choices at each step gives the stated result.
	\end{proof}
	%As an aside, the sunflower proof used the above approach without using the notion of pathological sets.  Instead they bounded the number of choices for S by roughly noting that any subset of S\cup W can't contain many edges because we're spread.
	\begin{claim}
		The number of $(S,W)\in \c{B}_{t}$ which are pathological is at most
		\[2(C/2)^{-k/2} |\cH|{r\choose t}{n-r\choose w}\]
	\end{claim}
	%We note that a much shorter version of this part of the argument can implicitly be found in \SP{Hamiltonian Squares}, but at least for myself I found it useful to explicitly write out all ofs the tiny calculations that they sweep under the rug. 
	\begin{proof}
		We identify these pairs by first specifying $S\in \cH$, then $S\cap W$, then $W\sm S$.  
		
		Note that $S$ and $S\cap W$ can be specified in at most $|\cH|\cdot {r\choose t}$ ways, and from now on we fix such a choice of $S$ and $S\cap W$.  It remains to specify $W\sm S$, which will be some element of ${V\sm S\choose w}$.  Thus it suffices to count the number of $W'\in {V\sm S\choose w}$ such that $(S,W')$ is both $k$-bad and pathological.%, which is equivalent to determining the probability that a randomly chosen $W'$ is both $k$-bad and pathological.
		
		For $W'\in {V\sm S\choose w}$, define \[\c{S}(W')=|\{S'\in \cH:S'\sub  (S\cup W'),\  |S'\cap S|\ge k\}|.\]  Observe that if $(S,W')$ is $k$-bad, then every edge $S'\sub (S\cup W')$ has $|S'\cap S|\ge k$ (since $|S'\cap S|\ge |S'\sm W'|$), so the $W'$ we wish to count satisfy \[\c{S}(W')=|\{S'\in \cH: S'\sub (S\cup W')|.\]  If $(S,W')$ is pathological, then this latter set has size at least $N$.  In total, if $\b{W}'$ is chosen uniformly at random from ${V\sm S\choose w}$, then
		\begin{equation}\Pr[(S,\b{W}')\tr{ is }k\tr{-bad and pathological}]\le \Pr[\c{S}(\b{W}')\ge N]\le \f{\E[\c{S}(\b{W}')]}{N},\label{eq:probW'}\end{equation}
		where this last step used Markov's inequality.  It remains to upper bound $\E[\c{S}(\b{W}')]$.
		
		Let \[m_j(S)=|\{S'\in \cH: |S\cap S'|=j\}|,\] and observe that for any $S'$ with $|S\cap S'|=j$, the number of $W'\in {V\sm S\choose w}$ with $S'\sub S\cup W'$ is exactly ${n-2r+j\choose w-r+j}$.  With this we see that
		\begin{align}\E[\c{S}(\b{W}')]=\sum_{j\ge k} m_j(S) \f{{n-2r+j\choose w-r+j}}{{n-r\choose w}}&=\sum_{j\ge k} m_j(S) \f{{w\choose r-j}}{{n-r\choose r-j}}=\f{{w+r\choose r}}{{n\choose r}}\sum_{j\ge k} m_j(S) \f{{w\choose r-j}}{{n-r\choose r-j}}\cdot \f{{n\choose w+r}}{{n-r\choose w}}.\label{eq:spreadExpectation}\end{align}
		Because $\cH$ is $(q;r,k)$-spread, we have for each $j\ge k$ in the sum that
		\begin{equation}m_j(S)\le M_j(S)\le q^j |\cH|.\label{eq:expectation1}\end{equation}
		For integers $x,y$, define the falling factorial $(x)_y:=x(x-1)\cdots(x-y+1)$.  With this we have
		\begin{align}
			\f{{w\choose r-j}}{{n-r\choose r-j}}\cdot \f{{n\choose w+r}}{{n-r\choose w}}&=\f{(w)_{r-j}}{(n-r)_{r-j}}\cdot \f{(n)_{r}}{(w+r)_r}\le \l(\f{w}{n-r}\r)^{r-j}\cdot \l(\f{n-r}{w}\r)^r=\l(\f{w}{n-r}\r)^{-j}\le (Cq/2)^{-j}, \label{eq:spreadBinomial}
		\end{align}
		where the first inequality used $w\le pn\le \half n\le n-r$, and the second inequality used
		\[w=pn-t\ge pn-r\ge pn/2=Cqn/2.\]
		Combining \eqref{eq:spreadExpectation}, \eqref{eq:expectation1}, and \eqref{eq:spreadBinomial} shows that
		\[\E[\c{S}(\b{W}')]\le \f{{w+r\choose r}}{{n\choose r}} |\cH|(C/2)^{-k}\cdot \sum_{j\ge k} (C/2)^{k-j}\le   \f{{w+r\choose r}}{{n\choose r}} |\cH|(C/2)^{-k}\cdot 2,\]
		where this last step used $C\ge 4$.  Plugging this into \eqref{eq:probW'} shows that the number of $W'\in {V\sm S\choose w}$ such that $(S,W')$ is $k$-bad and pathological is at most 
		
		\[ 2(C/2)^{-k} |\cH|\f{{w+r\choose r}}{{n\choose r} N}\cdot {n-r\choose w}= 2(C/2)^{-k/2}\cdot {n-r\choose w}.\]
		Combining this with the fact that there were $|\cH|\cdot {r\choose t}$ ways of choosing $S$ and $S\cap W$ gives the claim.
	\end{proof}

	In total $|\c{B}_{t}|$ is at most the sum of the bounds from these two claims.  Using this and $w=pn-t$ implies
	\begin{align*}\sum_{t\le r}|\c{B}_{t}|&\le\sum_{t\le r} 3 (C/2)^{-k/2} |\cH|{r\choose t}{n-r\choose pn-t}\\&= 3 (C/2)^{-k/2}|\cH| {n\choose pn},\end{align*}
	giving the desired result.
\end{proof}

\subsection{Auxiliary Lemmas}

To prove Theorem~\ref{thm:spreadMain}, we need to consider two special cases.  The first is when $\cH$ is $r$-uniform with $r$ relatively small.  In this case the following lemma gives effective bounds. 
\begin{lem}[\cite{Fractional}]\label{lem:small}
	Let $\cH$ be a $q$-spread $r$-bounded hypergraph on $V$ and $\al\in (0,1)$ such that $\al\ge 2rq$.  If $W$ is a set of size $\al|V|$ chosen uniformly at random from $V$, then the probability that $W$ does not contain an element of $\cH$ is at most
	\[2e^{-\al/(2rq)}.\] 
\end{lem}

The other special case we consider is the following.
\begin{lem}\label{lem:uniform}
	Let $\cH$ be an $r$-uniform $(q;r,1)$-spread hypergraph on $V$ and $\al\in(0,1)$ such that $\al\ge 4q$.  If $W$ is a set of size $\al|V|$ chosen uniformly at random from $V$, then the probability that $W$ does not contain an edge of $\cH$ is at most \[4q\al^{-1}+2e^{-\al|V|/4}.\]
\end{lem}
\begin{proof}
	Let $W'$ be a random set of $V$ obtained by including each vertex independently and with probability $\al/2$. Let $X=|\{S\in \cH: S\sub W'\}|$ and define $m_j(S)$ to be the number of $S'\in \cH$ with $|S\cap S'|=j$.  Note that $\E[X]=(\al/2)^r|\cH|$ and that
	\begin{align*}\Var(X)&\le (\al/2)^{2r}\sum_{S\in \cH}\sum_{S'\in \cH,\ S\cap S'\ne \emptyset} (\al/2)^{-|S\cap S'|} =(\al/2)^{2r}\sum_{S\in \cH}\sum_{j=1}^r (\al/2)^{-j} \cdot m_j(S)\\ 
		&\le(\al/2)^{2r}\sum_{S\in \cH}\sum_{j=1}^r (\al/2)^{-j} \cdot q^j |\cH| =(\al/2)^{2r}\sum_{j=1}^r (\al/2q)^{-j} |\cH|^2\\&= \E[X]^2 (\al/2q)^{-1} \sum_{j=1}^{r} (\al/2q)^{1-j}\le 4\E[X]^2 q\al^{-1}, 
	\end{align*} 
	where the second inequality used that $\cH$ being $(q;r,1)$-spread implies $m_j(S)\le q^j |\cH|$ for all $S\in \cH$ and $j\ge 1$, and the last inequality used $\al/2q\ge 2$.  By Chebyshev's inequality we have
	\[\Pr[X=0]\le \Var(X)/\E[X]^2\le 4q\al^{-1}.\]
	
	Lastly, observe that
	\begin{align*}\Pr[W\tr{ contains an edge of }\cH]&\ge \Pr[W'\tr{ contains an edge of }\cH\big| |W'|\le \al |V|]\\ &\ge \Pr[W'\tr{ contains an edge of }\cH]-\Pr[W'>\al |V|].\end{align*}
	By the Chernoff bound (see for example \cite{ProbMeth}) we have $\Pr[|W'|> \al |V|]\le 2e^{-\al |V|/4}$.  Note that $W'$ contains an edge of $\cH$ precisely when $X>0$, so the result follows from our analysis above.
\end{proof}

We conclude this subsection with a small observation.
\begin{lem}\label{lem:vertBound}
	If $\cH$ is an $r_1$-uniform $(q;r_1,\ldots,r_\ell)$-spread hypergraph on $V$, then $r_1\le eq|V|$.
\end{lem}
\begin{proof}
	Let $m=\max_{S\in \cH} d(S)$, i.e. this is the maximum multiplicity of any edge in $\cH$.  Then for any $S\in \cH$ with $d(S)=m$, we have
	\[m=M_{r_1}(S)\le q^{r_1}|\cH|\le q^{r_1}\cdot m{|V|\choose r_1}\le m(eq|V|/r_1)^{r_1},\]
	proving the result.
\end{proof}

\subsection{Putting the Pieces Together}
We now prove a technical version of Theorem~\ref{thm:spreadMain} with more explicit quantitative bounds.  Theorem~\ref{thm:spreadMain} will follow shortly (but not immediately) after proving this.

\begin{thm}\label{thm:spreadMainTech}
	Let $\cH$ be an $r_1$-uniform $(q;r_1,\ldots,r_\ell,1)$-spread hypergraph on $V$ and let $C\ge 8$ be a real number.  If $W$ is a set of size $2C\ell q|V|$ chosen uniformly at random from $V$, then
	\begin{equation}\Pr[W\tr{ contains an edge of }\cH]\ge1-6\ell^2 (C/4)^{-r_\ell/2}-40(C \ell)^{-1},\label{eq:spreadBound2}\end{equation}
	and for any $i$ with $4r_i\le C\ell $ we have
	\begin{equation}\Pr[W\tr{ contains an edge of }\cH]\ge 1-6\ell^2 (C/4)^{-r_{i}/2}-2e^{-C\ell/4r_{i}}.\label{eq:spreadBound1}\end{equation}
\end{thm}
%We note that the bound of \eqref{eq:spreadBound1} can be much stronger than the bound stated in Theorem~\ref{thm:spreadMain}.  Sometimes such quantiative bounds are needed for applications, and in particular this was the case in \cite{Sunflower} where the authors needed to show that $W$ contained an edge with probability greater than $1/k$, where $k$ is the number of edges in a sunflower.

%\footnote{Observe that this is somewhat reminiscent of the proof of the container lemma where we iteratively reduced the uniformity of our initial hypergraph while maintaining some ``nice'' properties.}
\begin{proof}
	%The intuitive idea of the argument is to iteratively take random sets $W_1,\ldots,W_\ell$ of size $Cq|V|$ from $V$.  If we let $\cH_i$ be the hypergraph after deleting $W_1\cup \cdots \cup W_{i-1}$ from $V$, then we would like to say that most $S\in \cH_i$ have $|S\sm W_i|\le r_{i+1}$.  This would imply at the final step that almost every edge of $\cH$ is covered by $\bigcup W_i$.  This will not happen exactly for each edge $S\in \cH_i$, but by Lemma~\ref{lem:spreadMain}, most $S\in \cH_i$ will have some $S'\sub S\cup W_i$ which has $|S'\sm W_i|\le r_i$.  We can thus achieve our goal by letting the edge set of $\cH_i$ consist of these $S'\sm W_i$ sets; though in order to keep our hypergraph uniform, we will actually take some set $A_S$ which contains $S'\sm W_i$.  We now move onto the technical details.
	
	Define $p:=Cq$ and $n:=|V|$.  We can assume $p\le \half$, as otherwise the result is trivial.  Let $W_1,\ldots W_{\ell-1}$ be chosen independently and uniformly at random from ${V\choose pn}$.  Throughout this proof we let $r_{\ell+1}=1$.

	Let $\cH_1=\cH$ and let $\phi_1:\cH_1\to \cH$ be the identity map.  Inductively assume we have defined $\cH_i$ and $\phi_i:\cH_i\to \cH$ for some $1\le i<\ell$. Let $\cH'_i\sub \cH_i$ be all the edges $S\in \cH_i$ such that $(S,W_{i})$ is $r_{i+1}$-good with respect to $\cH_{i}$.  Thus for each $S\in \cH'_i$, there exists an $S'\in \cH_i$ such that $S'\sub S\cup W_i$ and $|S'\sm W_i|\le r_{i+1}$.  Choose such an $S'$ for each $S\in \cH'_i$ and let $A_S$ be any subset of $S$ of size exactly $r_{i+1}$ that contains $S'\sm W_i$ (noting that $S'\sm W_i\sub S$ since $S'\sub S\cup W_i$).  Finally, define $\cH_{i+1}=\{A_S: S\in \cH_i'\}$ and $\phi_{i+1}:\cH_{i+1}\to \cH$ by $\phi_{i+1}(A_S)=\phi_i(S)$. 
	
	Intuitively, $\phi_i(A)$ is meant to correspond to  the ``original'' edge $S\in \cH$ which generated $A$.  More precisely, we have the following.
	\begin{claim}\label{cl:phi}
		For $i\le \ell$, the maps $\phi_i$ are injective and $A\sub \phi_i(A)$ for all $A\in \cH_i$.
	\end{claim}
	\begin{proof}
		This claim trivially holds at $i=1$.  Inductively assume the result has been proved through some value $i$.  Observe that in the process for generating $\c{H}_{i+1}$, we have implicitly defined a bijection $\psi:\cH_i'\to \cH_{i+1}$ through the correspondence $\psi(S)=A_S$.
		
		By construction of $\phi_{i+1}$, we have $\phi_{i+1}(A)=\phi_i(\psi^{-1}(A))$, so $\phi_{i+1}$ is injective since $\phi_i$ was inductively assumed to be injective and $\psi$ is a bijection.  Also be construction we have $A\sub \psi^{-1}(A)$, and by the inductive hypothesis we have $\psi^{-1}(A)\sub \phi_i(\psi^{-1}(A))=\phi_{i+1}(A)$.  This completes the proof.
	\end{proof}
	
	For $i<\ell$, we say that $W_i$ is \textit{successful} if $|\cH_{i+1}|\ge (1-\rec{2\ell})|\cH_{i}|$.  Note that $|\cH_{i+1}|=|\cH_{i}'|$, so this is equivalent to saying that the number of $r_{i+1}$-bad pairs $(S,W_i)$ with $S\in \cH_{i}$ is at most $\rec{2\ell}|\cH_{i}|$.
	\begin{claim}\label{cl:spread}
		For $i\le \ell$, if $W_1,\ldots,W_{i-1}$ are successful, then $\cH_i$ is $(2q;r_i,\ldots,r_{\ell},1)$-spread.
	\end{claim}
	\begin{proof}
		For a hypergraph $\cH'$, we let $M_j(A;\cH')$ denote the number of edges of $\cH'$ intersecting $A$ in at least $j$ vertices.  By Claim~\ref{cl:phi}, if $\{A_1,\ldots,A_t\}$ are the set of edges of $\cH_i$ which intersect some set $A$ in at least $j$ vertices, then $\{\phi_i(A_1),\ldots,\phi_i(A_t)\}$ is a set of $t$ distinct edges of $\cH$ intersecting $A$ in at least $j$ vertices.  Thus for all sets $A$ and integers $j$ we have  $M_j(A;\cH_i)\le M_j(A;\cH)$.  
		
		If $A$ is contained in an edge $A'$ of $\cH_i$, then by Claim~\ref{cl:phi} $A$ is contained in the edge $\phi_i(A')$ of $\cH$.  Thus $d_{\c{H}_i}(A)>0$ implies $d_{\c{H}}(A)>0$.  By assumption of $\cH$ being $(q;r_1,\ldots,r_\ell,1)$-spread, if $A$ is a set with $r_{i'}\ge |A|\ge r_{i'+1}$ for some integer $i'$ such that $d_{\cH_i}(A)>0$, and if $j$ is an integer satisfying $j\ge r_{i'+1}$, then our previous observations imply
		
		\begin{equation}M_j(A;\cH_i)\le M_j(A;\cH) \le q^j |\cH|.\label{eq:M}\end{equation}
		
		Because each of $W_1,\ldots,W_{i-1}$ were successful, we have \[|\cH_i|\ge \l(1-\rec{2\ell}\r)^i |\cH|\ge \l(1-\rec{2\ell}\r)^\ell |\cH|\ge \half |\cH|,\]
		where in this last step we used that $(1-1/(2x))^x$ is an increasing function for $x\ge 1$. Plugging $|\cH|\le 2|\cH_i|$ into \eqref{eq:M} shows that $\cH_i$ is $(2q;r_i,\ldots,r_{\ell},1)$-spread as desired.
	\end{proof}
	\begin{claim}\label{cl:success}
		For $i< \ell$,
		\[\Pr[W_{i}\tr{ is not successful }|\ W_1,\ldots,W_{i-1}\tr{ are successful}]\le 6\ell (C/4)^{-r_{i+1}/2}.\]
	\end{claim}
	\begin{proof}
		By construction $\cH_i$ is $r_i$-uniform.  Conditional on $W_1,\ldots,W_{i-1}$ successful, Claim~\ref{cl:spread} implies that $\cH_{i}$ is in particular $(2q;r_{i},r_{i+1})$-spread.  By hypothesis we have $p\le \half$ and $C/2\ge 4$, and by Lemma~\ref{lem:vertBound} applied to $\cH$ we have $2r_i\le pn$ since $C\ge 2e$.  Thus we can apply Lemma~\ref{lem:spreadMain} to $\cH_i$ (using $C/2$ instead of $C$), which shows that the expected number of $r_{i+1}$-bad pairs $(S,W_i)$ is at most $3(C/4)^{-r_{i+1}/2}|\cH_i|$.  By Markov's inequality, the probability of there being more than $\rec{2\ell} |\cH_i|$ total $r_{i+1}$-bad pairs is at most $6\ell (C/4)^{-r_{i+1}/2}$, giving the result.
	\end{proof}
	We are now ready to prove the result.  Let $W$ and $W'$ be sets of size $2\ell pn$ and $\ell pn$ chosen uniformly at random from $V$.  Observe that for any $1\le i\le \ell$, the probability of $W$ containing an edge of $\cH$ is at least the probability of $W_1\cup \cdots\cup W_{i-1} \cup W'$ containing an edge of $\cH$, and  this is at least the probability that $W'$ contains an edge of $\cH_{i}$ (since every edge of $\cH_i$ is an edge of $\cH$ after removing vertices that are in $W_1\cup \cdots \cup W_{i-1}$), so it suffices to show that this latter probability is large for some $i$.  
	
	By Proposition~\ref{prop:spreadEquiv}(a) and Claim~\ref{cl:spread}, the hypergraph $\cH_{i}$ will be $(2q)$-spread if $W_1,\ldots,W_{i-1}$ are all successful.  If $i$ is such that $C\ell \ge 4r_i$, then by Claim~\ref{cl:success} and Lemma~\ref{lem:small} the probability that $W_1,\ldots,W_{i-1}$ are all successful and $W'$ contains an edge of $\cH_i$ is at least
	\[1-6\ell^2 (C/4)^{-r_{i}/2}-2e^{-C\ell/4r_{i}},\]
	giving \eqref{eq:spreadBound1}.
	
	Alternatively, the probability that $W'$ contains an edge of $\cH_{\ell}$ can be computed using Lemma~\ref{lem:uniform}, which gives that the probability of success is at least
	\[1-6\ell^2 (C/4)^{-r_\ell/2}-16(C \ell)^{-1}-2e^{-C\ell q n/4}.\]
	Using $qn\ge e^{-1}r_1\ge 1/3$ from Lemma~\ref{lem:vertBound} together with $e^{-x}\le x^{-1}$ gives \eqref{eq:spreadBound2}
	as desired.
\end{proof}

We now use this to prove Theorem~\ref{thm:spreadMain}.
\begin{proof}[Proof of Theorem~\ref{thm:spreadMain}]
	There exists a large constant $K'$ such that if\footnote{We consider $\log(\ell+1)$ as opposed to $\log(\ell)$ to guarantee that this is a positive number for all $\ell\ge 1$.} $r_\ell\ge K' \log (\ell+1)$, then the result follows from \eqref{eq:spreadBound2}.  If this does not hold and if $r_1>K'\log(\ell+1)$, then there exists some $I\ge 2$ such that $r_{I-1}> K'\log(\ell+1)\ge r_I$.  If $r_I=K'\log(\ell+1)$, then the result follows from \eqref{eq:spreadBound1} with $i=I$ provided $C$ is sufficiently large in terms of $K'$.  Otherwise we define a new sequence of integers $r'_1,\ldots,r'_{\ell+1}$ with $r'_i=r_i$ for $i<I$, $r'_I=K'\log(\ell+1)$, and $r'_i=r_{i-1}$ for $i>I$.  It is not hard to see that $\cH$ is $(q;r_1',\ldots,r'_{\ell+1},1)$-spread, so the result follows\footnote{The bound of \eqref{eq:spreadBound1} now uses $\ell+1$ instead of $\ell$ throughout because we are working with the $r'_i$ sequence, but this does not affect the final result.} from \eqref{eq:spreadBound1} with $i=I$.
	
	It remains to deal with the case $r_1\le K'\log(\ell+1)$.  Because $\ell\le r_1$, this can only hold if $r_1\le K''$ for some large constant $K''$.  In this case we can apply Lemma~\ref{lem:small} to give the desired result by choosing $K_0$ sufficiently large in terms of $K''$.
\end{proof} 
	
\section{Concluding Remarks}\label{sec:con}
With a very similar proof one can prove the following non-uniform analog of Theorem~\ref{thm:spreadMain}.

\begin{thm}\label{thm:nonUniform}
	Let $\cH$ be a $(q;r_1,\ldots,r_\ell,1)$-spread hypergraph on $V$ and define $s=\min_{S\in \cH}|S|$.  Assume that there exists a $K$ such that $r_1\le K q|V|$, and such that for all $i$ with $r_{i}>s$ we have $\log r_i\le K r_{i+1}$.  Then there exists a constant $K_0$ depending only on $K$ such that if $r_\ell\le \max\{s,K_0\log(\ell+1)\}$ and $C\ge K_0$, then a set $W$ of size $C\ell q|V|$ chosen uniformly at random from $V$ satisfies
	\[\Pr[W\tr{ contains an edge of }\cH]\ge 1-\f{K_0}{C\ell}.\]
\end{thm}
Observe that if $\cH$ is $r_1$-uniform then this reduces to Theorem~\ref{thm:spreadMain} with the additional constraint that $r_1\le Kq|V|$ for some $K$.  By Lemma~\ref{lem:vertBound}, this extra condition is always satisfied for uniform hypergraphs with $K=e$. We note that Theorem~\ref{thm:nonUniform} together with Proposition~\ref{prop:spreadEquiv}(b) implies Theorem~\ref{thm:Fractional}.  We briefly describe the details on how to prove this.

\begin{proof}[Sketch of Proof]
	We first adjust the statement and proof of Lemma~\ref{lem:spreadMain} to allow $\cH$ to be $r$-bounded.  To do this, we partition $\cH$ into the uniform hypergraphs $\cH_{r'}=\{S\in \cH:|S|=r'\}$, and word for word the exact same proof\footnote{The $\cH_{r'}$ hypergraphs may not be spread, but they still have the property that $m_j(S)\le q^j|\cH|$ for all $S\in \cH_{r'}\sub \cH$, and this is the only point in the proof where we used that $\cH$ is spread.} as before shows that the number of $k$-bad pairs using $S\in \cH_{r'}$ is at most $3 (C/2)^{-k/2}|\cH| {n\choose pn}$.  We then add these bounds over all $r'$ to get the same bound as in Lemma~\ref{lem:spreadMain} multiplied by an extra factor of $r$.  With regards to the other lemmas, one no longer needs Lemma~\ref{lem:vertBound} due to the $r_1\le K q|V|$ hypothesis, and Lemmas~\ref{lem:small} and \ref{lem:uniform} are fine as is (in particular, Lemma~\ref{lem:uniform} still requires $\cH$ to be uniform).
	
	For the main part of the proof, instead of choosing $A_S$ to be a subset of $S$ of size exactly $r_i$, we choose it to have size at most $r_i$ and at least $\min\{r_i,s\}$.  With this $\cH_i$ will be uniform if $r_i\le s$, and otherwise when we apply the non-uniform version of Lemma~\ref{lem:spreadMain} our error term will have an extra factor of $r_i\le e^{Kr_{i+1}}$, with this inequality holding by our hypothesis for $r_i>s$.  This term will be insignificant compared to $(C/2)^{-r_{i+1}/2}$ provided $C$ is large in terms of $K$.  
	
	If $r_\ell \le K' \log(\ell+1)$ for some large $K'$ depending on $K$, then as in the proof of Theorem~\ref{thm:spreadMain} we can assume $r_I=K'\log(\ell+1)$ for some $I$ and conclude the result as before.  Otherwise $r_\ell\le s$ by hypothesis, so $\cH_\ell$ will be uniform and we can apply Lemma~\ref{lem:uniform} to conclude the result.
\end{proof}

Another extension can be made by not requiring the same ``level of spreadness'' throughout $\cH$.
\begin{defn}
	Let $0<q_1,\ldots,q_{\ell-1}\le 1$ be real numbers and $r_1>\cdots >r_\ell$ positive integers. We say that a hypergraph $\cH$ on $V$ is \textit{$(q_1,\ldots,q_{\ell-1};r_1,\ldots,r_\ell)$-spread} if $\cH$ is non-empty, $r_1$-bounded, and if for all $A\sub V$ with $d(A)>0$ and $r_{i}\ge|A|\ge r_{i+1}$ for some $1\le i< \ell$, we have for all $j\ge r_{i+1}$ that
	\[M_j(A):=|\{S\in \cH:|A\cap S|\ge j\}|\le q_i^j |\cH|.\]
\end{defn}
\iffalse 
\begin{defn}
	Let $0<q_1,\ldots,q_{\ell-1}\le 1$ be real numbers and $r_1>\cdots >r_\ell$ positive integers. We say that a hypergraph $\cH$ on $V$ is \textit{$(q_1,\ldots,q_{\ell-1};r_1,\ldots,r_\ell)$-spread} if $\cH$ satisfies the same conditions as in Definition~\ref{def:2} except $q^j$ in \eqref{eq:q} is replaced by $q_i^j$.
\end{defn}
\fi

Different levels of spread was also considered in \cite{Sunflower}.  Here one can prove the following.

\begin{thm}\label{thm:levels}
	Let $\cH$ be a $(q_1,\ldots,q_{\ell};r_1,\ldots,r_\ell,1)$-spread hypergraph on $V$ and define $s=\min_{S\in \cH}|S|$.  Assume that there exists a $K$ such that for all $i$ we have $r_i\le K q_i|V|$, and that for all $i$ with $r_{i}>s$ we have $\log r_i\le K r_{i+1}$.  Then there exists a constant $K_0$ depending only on $K$ such that if $r_\ell\le \max\{s,K_0\log(\ell+1)\}$ and if $C\ge K_0$, then a set $W$ of size $C\sum q_i|V|$ chosen uniformly at random from $V$ satisfies
	\[\Pr[W\tr{ contains an edge of }\cH]\ge 1-\f{K_0 \log (\ell+1)}{CL},\]
	where $L:=\sum_i q_i/\max_i q_i$.
\end{thm}
Note that $\sum q_i\le \ell \max q_i$, so we have $L\le \ell$ with equality if $q_i=q_j$ for all $i,j$.

\begin{proof}[Sketch of Proof]
	We now choose our random sets $W_i$ to have sizes $Cq_i|V|$ and $W'$ to have size $C\sum q_i |V|=C(L\cdot \max q_i)|V|$.   With this any of the $\cH_i$ could be at worst $(2\max q_i)$-spread if each $\cH_{i'}$ was successful, so in this case when we apply Lemma~\ref{lem:small} with $W'$ we end up getting a probability of roughly $1-e^{-CL/r_i}$ of containing an edge.  From  this quantity we should subtract roughly $\ell^2 C^{-r_i}$, since this is the probability that some $\cH_{i'}$ is unsuccessful.  If $r_i=K'\log(\ell+1)$ for some large constant $K'$ then this gives the desired bound.  Otherwise we can basically assume $r_\ell> K'\log(\ell+1)$ and apply Lemma~\ref{lem:uniform} to $\cH_\ell$ to get a probability of roughly $1-(CL)^{-1}$, which also gives the result after subtracting $\ell^2 C^{-r_\ell}$ to account for some $\cH_{i'}$ being unsuccessful.
\end{proof}

Recently Frieze and Marbach~\cite{FriezeMarbach} developed a variant of Theorem~\ref{thm:Fractional} for rainbow structures in hypergraphs.  We suspect that straightforward generalizations of our proofs and those of \cite{FriezeMarbach} should give an analog of Theorem~\ref{thm:spreadMain} (as well as Theorems~\ref{thm:nonUniform} and \ref{thm:levels}) for the rainbow setting.

\textbf{Acknowledgments:} We thank Bhargav Narayanan for looking over an earlier draft.

\bibliographystyle{abbrv}
\bibliography{Spread}
	
\end{document}